\documentclass[a4paper,11pt]{amsart}

\usepackage[utf8]{inputenc}
\usepackage[T1]{fontenc}
\usepackage[english]{babel}
\usepackage{amsmath}
\usepackage{amsthm}
\usepackage{mathtools}
\usepackage{amsfonts}
\usepackage{amssymb}

\usepackage{complexity}

\usepackage{hyperref}
\usepackage{url}

\usepackage{mathdots}

\usepackage{xcolor,graphicx,ifthen}
\usepackage{subfigure}

\usepackage{nicefrac}

\usepackage[makeroom]{cancel}

\input{makros.sty}

\def\authorinformation{
	Institute of Mathematics, Würzburg University, Emil-Fischer-Str. 40, 97074 Würzburg \\E-mail: \text{\href{mailto:katja.moenius@mathematik.uni-wuerzburg.de}{katja.moenius@mathematik.uni-wuerzburg.de}}
}

\date{\today}

\title{\thetitle}

\makeatletter
\renewcommand*\@maketitle{%
	\normalfont\normalsize
	\@adminfootnotes
	\@mkboth{\@nx\shortauthors}{\@nx\shorttitle}%
	\ifx\@empty\@date \else {\vskip 1em \vtop{\flushright\small\@date\@@par}} \vspace{5mm}\fi
	\global\topskip42\p@\relax 
	\Large\@settitle
	\ifx\@empty\authors \else \@setauthors \fi
	\ifx\@empty\authorinformation
	\else
	\baselineskip9\p@
	\vtop{\centering{\footnotesize\textit\authorinformation\@@par}%
		\global\dimen@i\prevdepth}\prevdepth\dimen@i
	\fi
	\ifx\@empty\@dedicatory
	\else
	\baselineskip18\p@
	\vtop{\centering{\footnotesize\itshape\@dedicatory\@@par}%
		\global\dimen@i\prevdepth}\prevdepth\dimen@i
	\fi
	\@setabstract
	\normalsize
	\if@titlepage
	\newpage
	\else
	\dimen@34\p@ \advance\dimen@-\baselineskip
	\vskip\dimen@\relax
	\fi
} 
\renewcommand*\@adminfootnotes{%
	\let\@makefnmark\relax  \let\@thefnmark\relax
	\ifx\@empty\@subjclass\else \@footnotetext{\@setsubjclass}\fi
	\ifx\@empty\@keywords\else \@footnotetext{\@setkeywords}\fi
	\ifx\@empty\thankses\else \@footnotetext{%
		\def\par{\let\par\@par}\@setthanks}%
	\fi
}
\makeatother

\newcommand{\setupPaper}[2]{
	\def\thetitle{#1}
	\def\thecreationdate{#2}
}

\setupPaper{Eigenvalues of zero-divisor graphs of finite commutative rings}{09.09.2019}

\usepackage{float}

\def\spec{\operatorname{spec}}
\def\rank{\operatorname{rank}}
\def\zg{\operatorname{Z}}
\def\zgl{\operatorname{Z_L}}
\def\ug{\operatorname{U}}
\def\eg{\operatorname{E\Gamma}}
\def\ann{\operatorname{ann}}

\author{Katja Mönius}

\begin{document}
	
\begin{abstract}
	We investigate eigenvalues of the zero-divisor graph $\Gamma(R)$ of  finite commutative rings $R$ and study the interplay between these eigenvalues, the ring-theoretic properties of $R$ and the graph-theoretic properties of $\Gamma(R)$. The graph $\Gamma(R)$ is defined as the graph with vertex set consisting of all non-zero zero-divisors of $R$ and adjacent vertices $x,y$ whenever $xy = 0$. We provide formulas for the nullity of $\Gamma(R)$, i.e. the multiplicity of the eigenvalue 0 of $\Gamma(R)$. Moreover, we precisely determine the spectra of $\Gamma(\mathbb Z_p \times \mathbb Z_p \times \mathbb Z_p)$ and $\Gamma(\mathbb Z_p \times \mathbb Z_p \times \mathbb Z_p \times \mathbb Z_p)$ for a prime number $p$. We introduce a graph product $\times_{\Gamma}$ with the property that $\Gamma(R) \cong \Gamma(R_1) \times_{\Gamma} \ldots \times_{\Gamma} \Gamma(R_r)$ whenever $R \cong R_1 \times \ldots \times R_r.$
	With this product, we find relations between the number of vertices of the zero-divisor graph $\Gamma(R)$, the compressed zero-divisor graph, the structure of the ring $R$ and the eigenvalues of $\Gamma(R)$.
\end{abstract}
\maketitle

\section{Introduction}
Let $R$ be a finite commutative ring with $1 \neq 0$ and let $Z(R)$ denote its set of zero-divisors. As introduced by Anderson~and~Livingston~\cite{Anderson1999} in 1999, the \textit{zero-divisor graph} $\Gamma(R)$ is defined as the graph with vertex set $Z^*(R) = Z(R) \backslash \{0\}$ where two vertices $x,y$ are adjacent if and only if $xy = 0$. The aim of considering these graphs is to study the interplay between graph theoretic properties of $\Gamma(R)$ and the ring properties of $R$. In order to simplify the representation of $\Gamma(R)$ it is often useful to consider the so-called \textit{compressed zero-divisor graph} $\Gamma_E(R)$. This graph was first introduced by Mulay~\cite{Mulay2007} and further studied in \cite{Spiroff2011, Weber2011, Anderson2016, Pirzada2018}. For an element $r \in R$ let $[r]_R = \{s \in R \mid \ann_R(r) = \ann_R(s) \}$ and $R_E = \{[r]_R \mid r \in R\}$.  Then $\Gamma_E(R)$ is defined as the graph $\Gamma(R_E)$. Note that $[0]_R = \{0\}, [1]_R = R \backslash Z(R)$ and $[r]_R \subseteq Z(R) \backslash \{0\}$ for every $r \in R \backslash ([0]_r \cup [1]_R)$. The notations are adopted from Spiroff~and~Wickham~\cite{Spiroff2011}.

The \textit{spectrum} of a graph $G$ is defined as the multi-set of eigenvalues, i.e. the roots of the characteristic polynomial of the adjacency matrix $A(G)$. The aim of studying eigenvalues of graphs is to find relations between those values and structural properties of the graph. The author refers to \cite{Brouwer2012} for a good introduction to spectral graph theory. The \textit{nullity} $\eta(G)$ of a graph $G$ is defined as the multiplicity of the eigenvalue 0 of $G$. Obviously, we have that $\eta(G) = \dim A(G) - \rank A(G)$. Background and further results on the nullity of graphs are summarized in \cite{Gutman2011}.
Within spectral graph theory, most graphs are considered to be \textit{simple}, i.e. to be undirected finite graphs without loops or multiple edges. By definition, $\Gamma(R)$ has no multiple edges, and we can easily see  that $\Gamma(R)$ is undirected if and only if $R$ is commutative. Moreover, as already proven by Anderson~and~Livingston~\cite[Theorem 2.2]{Anderson1999}, the graph $\Gamma(R)$ is finite if and only if $R$ is finite or an integral domain. In the latter case, though, $R$ has no zero-divisors at all and is just the empty graph. Hence, all our rings are assumed to be finite and commutative. However, in contrast to the original definition of Anderson~and~Livingston~\cite{Anderson1999}, we do not want to eliminate potential loops of our zero-divisor graphs since these loops provide important information about the structure of the ring $R$.

 In order to determine the eigenvalues of a graph, it often can be useful to consider graph products. For example, in \cite{Akhtar2009} the spectra of unitary Cayley graphs of finite rings could easily be determined by observing that these graphs are isomorphic to direct products of unitary Cayley graphs of finite local rings. The \textit{direct product} $G_1 \times  G_2$ of graphs $G_1$ and $G_2$ is defined as the graph with vertex set $V(G_1) \times V(G_2)$ where two vertices $(v_1,v_2), (v_1',v_2') \in V(G_1) \times V(G_2)$ are adjacent in $G_1 \times G_2$ if and only if $v_1$ is adjacent to $v_1'$ in $G_1$ and $v_2$ is adjacent to $v_2'$ in  $G_2$.  It is well-known that the adjacency matrix of the direct product $G_1 \times G_2$ equals the Kronecker product $A(G_1) \otimes A(G_2)$. Therefore, if $\lambda_i$ resp. $\mu_i$ are the eigenvalues of $G_1$ resp. $G_2$, the eigenvalues of $G_1 \times G_2$ are exactly the products $\lambda_i \mu_j$. Moreover, the \textit{complete product} $G_1 \nabla G_2$ is defined to have vertex set $V(G_1) \cup V(G_2)$ and $v_1$ and $v_2$ are adjacent in $G_1 \nabla G_2$ if and only if either $v_1 \in V_1$ and $v_2 \in V_2$ or $v_1$ is adjacent to $v_2$ in $G_1$ resp. $G_2$. For $v_1 \in V(G_1)$, $v_2 \in V(G_2)$ the \textit{point identification} $G_1 \bullet G_2$ arises from setting $v_1 = v_2$. If $v \in V(G_1)$ and $v \in V(G_2)$, we write $G_1 \overset{v}{\bullet} G_2$ in order to make clear that the graphs were coalesced at $v$.

\smallskip
In this paper, we study the interplay between graph-theoretic properties of the zero-divisor graph $\Gamma(R)$, the spectrum of $\Gamma(R)$ and the ring properties of $R$. By now, surprisingly little is known about the eigenvalues and adjacency matrices of zero-divisor graphs. First research in this direction was done by Sharma~et.~al.~\cite{Sharma2011} in 2011. They made some observations on the adjacency matrices and eigenvalues of the graphs $\Gamma(\mathbb Z_p \times \mathbb Z_p)$ and $\Gamma(\mathbb Z_p[i] \times \mathbb Z_p[i])$. Further results were found by Young~\cite{Young2015} in 2015 and independently by Surendranath~Reddy~et.~al.~\cite{Reddy2017} in 2017. Both studied the graphs $\Gamma(\mathbb Z_n)$ and precisely determined the eigenvalues of $\Gamma(\mathbb Z_p), \Gamma(\mathbb Z_{p^2}),\Gamma(\mathbb Z_{p^3})$ and $\Gamma(\mathbb Z_{p^2q})$ for $p$ and $q$ being prime numbers. Other recent papers on that topic are \cite{Tirop2019, Omondi2018}. Note that in most of these papers the corresponding zero-divisor graphs were also considered with loops.

Our main approach is the following: since $R$ is a finite ring, it can be written as $R \cong R_1 \times \ldots \times R_r$, where each $R_i$ is a finite local ring. A proof for this and further results within the theory of finite commutative rings can be found in~\cite{Bini2002}. In Section~\ref{sec:products} we introduce a graph product $x_{\Gamma}$ with the property that $$\Gamma(R) \cong \Gamma(R_1) \times_{\Gamma} \ldots \times_{\Gamma} \Gamma(R_r)$$ whenever $R \cong R_1 \times \ldots \times R_r.$ With this graph product, in Section~\ref{sec:general} we find a relation between the number of vertices of $\Gamma_E(R)$ and the property of $R$ to be a local ring. Moreover, we derive formulas for the number of vertices of the zero-divisor graph $\Gamma(R)$ resp. the compressed zero-divisor graph $\Gamma_E(R)$ in terms of the local rings $R_i$. From these formulas, we can deduce a lower bound for the nullity of $\Gamma(R)$.  In Section~\ref{sec:integers} we restrict our considerations to rings which are isomorphic to direct products of rings of integers modulo $n$, i.e. $R \cong  \mathbb Z_{{p_1}^{t_1}} \times \ldots \times \mathbb Z_{{p_r}^{t_r}}$ for (not necessarily distinct) prime numbers $p_i$ and positive integers $r,t_i$. For these rings, we find the exact nullity of $\Gamma(R)$ and present an easy approach to determine also the non-zero eigenvalues of $\Gamma(R)$. For example, we precisely determine the spectra of $\Gamma(\mathbb Z_p \times \mathbb Z_p \times \mathbb Z_p)$ and $\Gamma(\mathbb Z_p \times \mathbb Z_p \times \mathbb Z_p \times \mathbb Z_p)$ in terms of a prime number $p$. We also provide the characteristic polynomials of $\Gamma(\mathbb Z_{p^2} \times \mathbb Z_p)$ and $\Gamma(\mathbb Z_p \times \mathbb Z_p \times \mathbb Z_q)$ for primes $q \neq p$. This generalizes the results of Sharma~et.~al.~\cite{Sharma2011}, Young~\cite{Young2015} and Surendranath~Reddy~et.~al.~\cite{Reddy2017}.

Throughout this paper, we denote edges as sets of two vertices. For a graph $G$, we write $A(G)$ for the adjacency matrix of $G$ , $V(G)$ for the set of vertices of $G$ and $\chi_G(x) = \det(xI - A(G))$ for the characteristic polynomial of $G$. If $\lambda$ is an eigenvalue of $G$ of multiplicity $x$, then we denote this by $\lambda^{[x]}$. The number of elements in a set $S$ is denoted by $\#S$, and $\varphi$ denotes Euler's totient function. For the set of units of a ring $R$, we write $U(R)$.

\section{Products of zero-divisor graphs} \label{sec:products}
Let $R \cong R_1 \times \ldots \times R_r$ be a ring, where each $R_i$ is a finite local ring. Note that in this case $\#R_i = p_i^{t_i}$ for some prime numbers $p_i$ and $t_i \in \mathbb N$. Our aim is to define a graph product $\times_{\Gamma}$ such that $$\Gamma(R) \cong \Gamma(R_1) \times_{\Gamma} \ldots \times_{\Gamma} \Gamma(R_r)$$ whenever $$R \cong R_1 \times \ldots \times R_r.$$ 
Since two vertices $(v_1,\ldots,v_r),(v_1',\ldots,v_r') \in R_1 \times \ldots \times R_r$ are adjacent in $\Gamma(R_1 \times \ldots \times R_r)$ if and only if $v_i,v_i' \in Z^*(R_i)$ or either $v_i = 0$ or $v_i'=0$, our idea is the following:  we first add the vertex $0 \in R_i$ and the units of $R_i$ to the vertices of each zero-divisor graph $\Gamma(R_i)$, as well as edges from $0$ to every other vertex. Then, we take the direct product of these somehow \textit{extended zero-divisor graphs}, each of which we will denote by $\eg(R_i)$, which yields the extended zero-divisor graph $\eg(R)$. Finally, by removing the vertex $0 \in R$ with all its edges, as well as all units of $R$, we end up with the zero-divisor graph $\Gamma(R)$.

To formalize this, we define the \textit{unit graph} $\ug(R_i)$ of $R_i$ as the graph with vertex set $U(R_i)$ and empty edge set. Moreover, let $\zg(R_i)$ resp. $\zgl(R_i)$ be the \textit{zero graph} with vertex set $\{0\}$ (where $0 \in R_i$) and empty edge set resp. edge set $\{ \{0,0 \} \}$ (i.e. both graphs consist of one vertex only, and, in contrast to $\zg(R)$, the graph $\zgl(R)$ also has a loop at that vertex; we need this distinction for our result in Section~\ref{sec:relation}). Now, the extended zero-divisor graph $\eg (R_i)$ is given by $$\eg(R_i) = \big (\Gamma(R_i) \nabla \zg(R_i) \big) \overset{\{0\}}{\bullet} \big(\ug(R_i) \nabla \zgl(R_i) \big),$$  and we have that \begin{equation*}
	\begin{split}
	\Big(\Gamma(R) \cup \big( \ug(R_1) \times \ldots \times \ug(R_r) \big) \Big) \nabla \zg(R) \cong 
	 \eg(R_1) \times \ldots \times \eg(R_r).
	\end{split}
\end{equation*}
 Hence, we define the associative product $\times_{\Gamma}$ by
 \begin{equation*}
 	\begin{split}
 	\Gamma(R_1) &\times_{\Gamma} \Gamma(R_2) := \\
 	& \Big( \eg(R_1) \times \eg(R_2)\Big) \backslash \Big(V \big(\zg(R_1 \times R_2) \big) \cup V \big(\ug(R_1 \times R_2) \big) \Big),
 	\end{split}
 \end{equation*} where $G\backslash\{v\}$ denotes the graph $G$ without the vertex $v \in V(G)$ and all its adjacent edges.
Note that $\zg(R_1 \times R_2) \cong \zg(R_1) \times \zg(R_2)$ and $\ug(R_1 \times R_2) \cong \ug(R_1) \times \ug(R_2)$.
The product $\times_{\Gamma}$ is illustrated in the following example:

\begin{example}
	Let $R = \mathbb Z_8 \times \mathbb Z_4$. Figure~\ref{fig:1} shows the zero-divisor graphs $\Gamma(\mathbb Z_8)$ and $\Gamma(\mathbb Z_4)$ and Figure~\ref{fig:2} the extended zero-divisor graphs $\eg(\mathbb Z_8)$ and $\eg(\mathbb Z_4)$. In Figure~\ref{fig:3} we see the direct product $\eg(\mathbb Z_8) \times \eg(\mathbb Z_4) \cong \eg(\mathbb Z_8 \times \mathbb Z_4)$ and Figure~\ref{fig:4} finally illustrates the graph product $\Gamma(\mathbb Z_8) \times_{\Gamma} \Gamma(\mathbb Z_4) \cong \Gamma(\mathbb Z_8 \times \mathbb Z_4)$ arising from removing the vertices $(0,0)$ and $V(\ug(\mathbb Z_8 \times \mathbb Z_4))$ from the graph $\eg(\mathbb Z_8) \times \eg(\mathbb Z_4)$.
\end{example}

\begin{figure}[H]
	\subfigure[$\Gamma(\mathbb Z_8)$]{\includegraphics[width=0.60\textwidth]{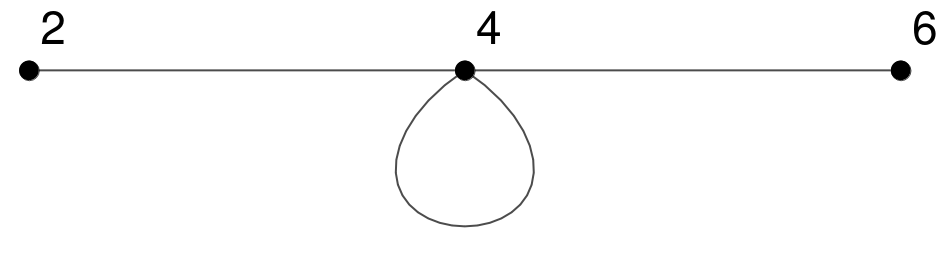}}
	\subfigure[$\Gamma(\mathbb Z_4)$]{\includegraphics[width=0.20\textwidth]{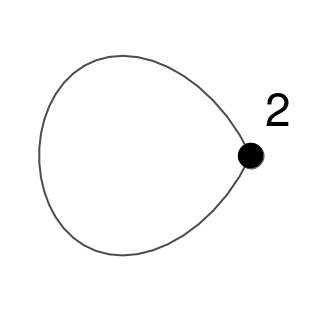}}
	\caption{Zero-divisor graphs $\Gamma(\mathbb Z_8)$ and $\Gamma(\mathbb Z_4)$.}
	\label{fig:1}
\end{figure}

\begin{figure}[H]
	\subfigure[$\eg(\mathbb Z_8)$]{\includegraphics[width=0.45\textwidth]{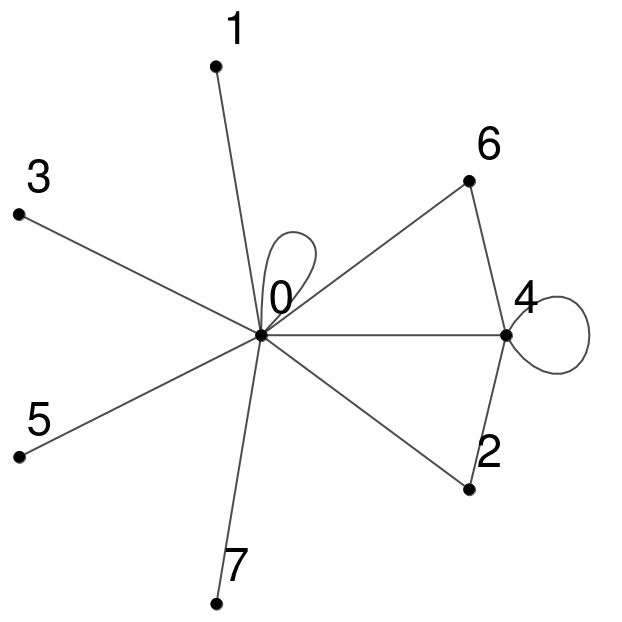}}
	\subfigure[$\eg(\mathbb Z_4)$]{\includegraphics[width=0.45\textwidth]{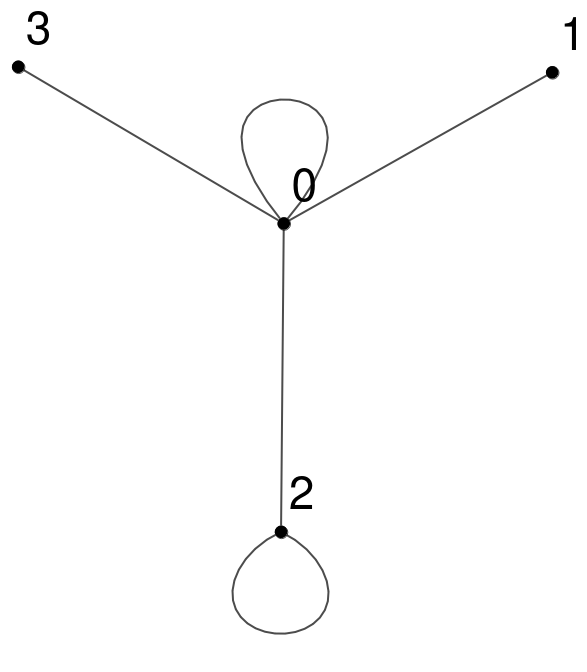}}
	\caption{Extended zero-divisor graphs $\eg(\mathbb Z_8)$ and $\eg(\mathbb Z_4)$.}
	\label{fig:2}
\end{figure}

\begin{figure}[H]
	\includegraphics[width=0.9\textwidth]{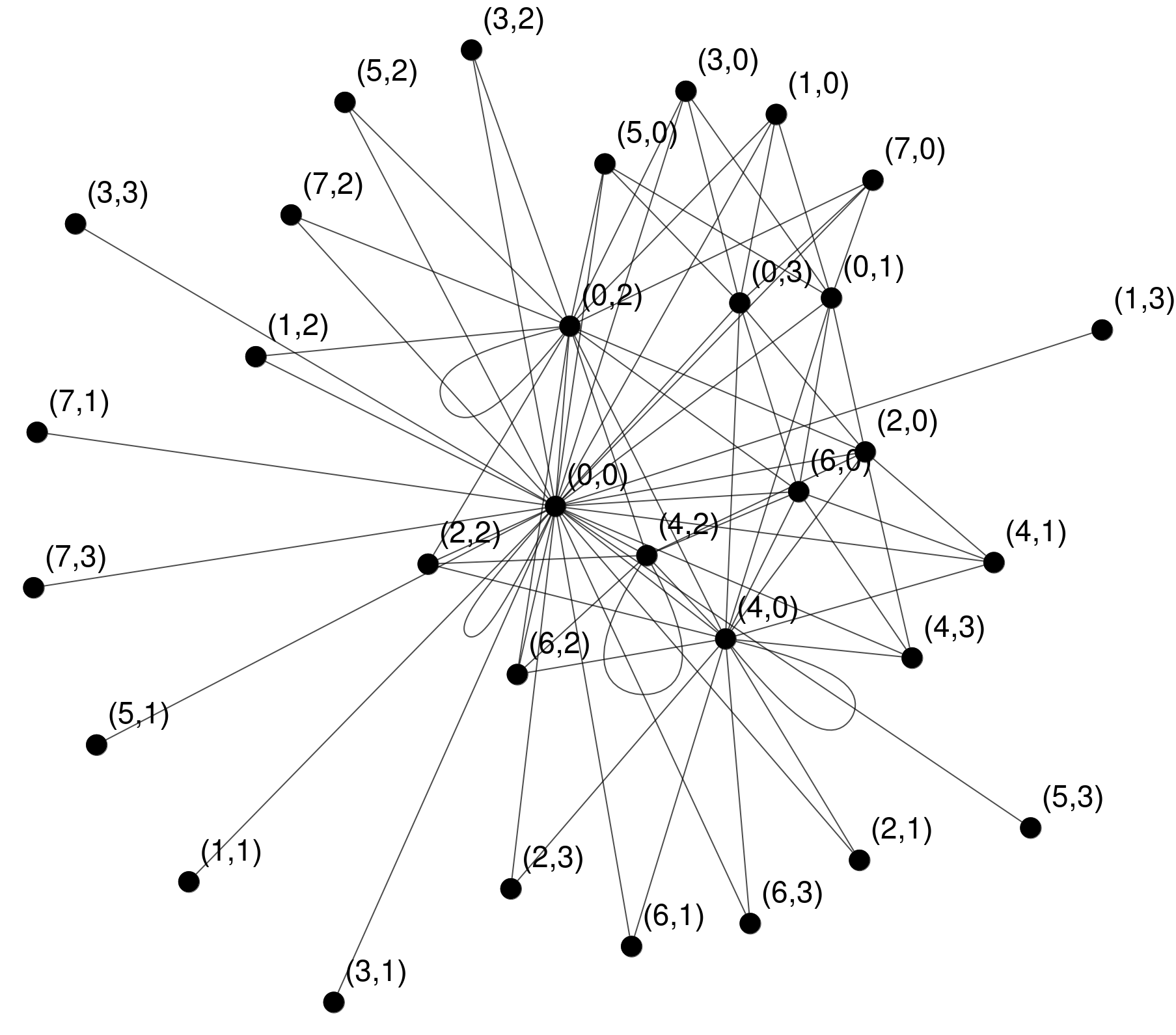}
	\caption{Direct product $\eg(\mathbb Z_8) \times \eg(\mathbb Z_4) \cong \eg(\mathbb Z_8 \times \mathbb Z_4)$.}
	\label{fig:3}
\end{figure}

\begin{figure}[H]
	\includegraphics[width=0.9\textwidth]{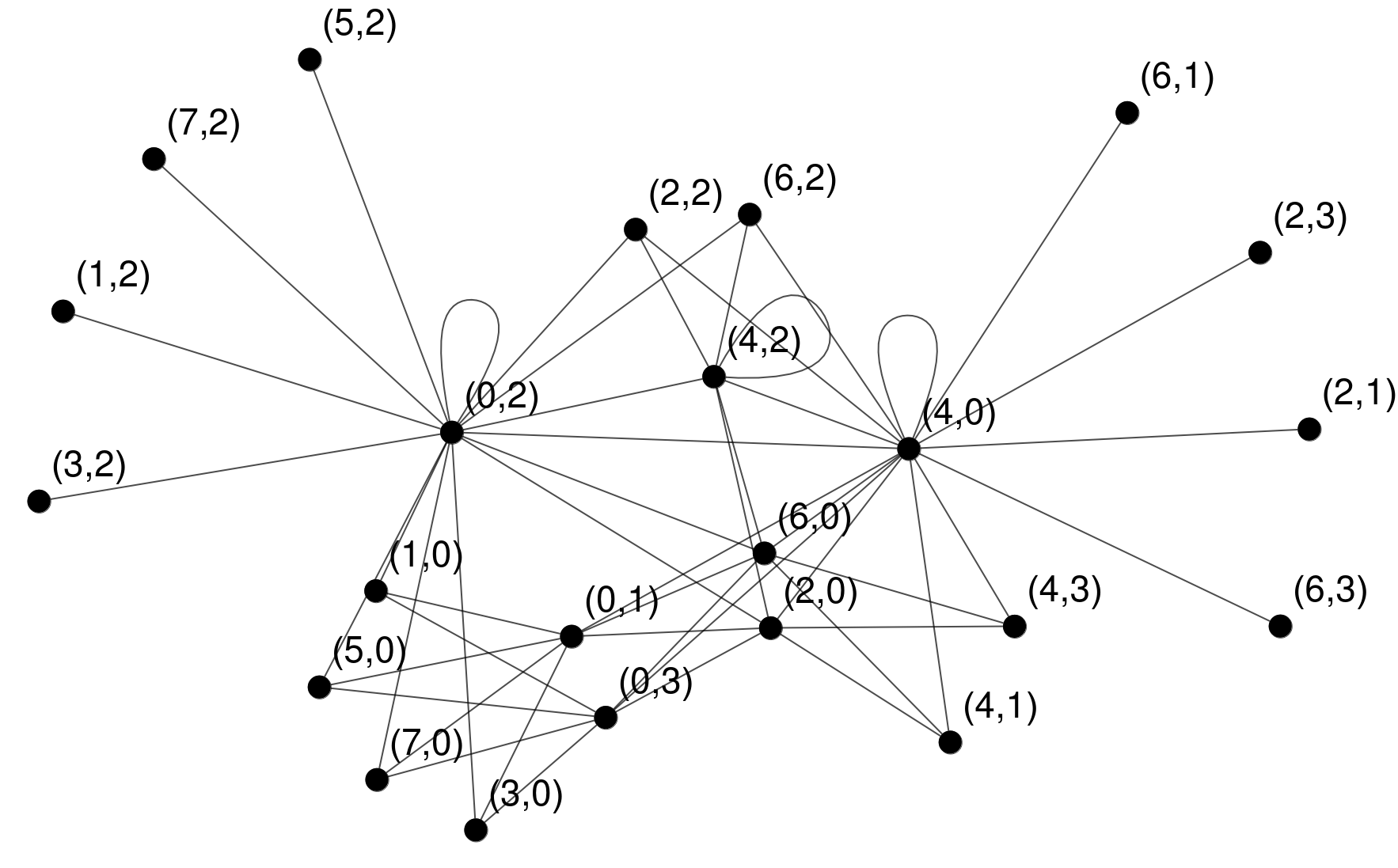}
	\caption{Zero-divisor graph $\Gamma(\mathbb Z_8 \times \mathbb Z_4) \cong \Gamma(\mathbb Z_8) \times_{\Gamma} \Gamma(\mathbb Z_4)$.}
	\label{fig:4}
\end{figure}

The same also holds for the compressed zero-divisor graph, i.e. we have that $\Gamma_E(R) \cong \Gamma_E(R_1) \times_{\Gamma} \ldots \times_{\Gamma} \Gamma_E(R_r)$ whenever $R_E \cong {R_1}_E \times \ldots \times {R_r}_E$.
In Section~\ref{sec:relation} we deduce a relation between the characteristic polynomial of $\Gamma(R)$ and the one of the extended zero-divisor graph $\eg(R)$.

\section{Nullity of zero-divisor graphs of finite commutative rings} \label{sec:general}
The following theorem follows directly from the construction of the product~$\times_{\Gamma}$:
\begin{theorem} \label{thm:verticesZDG}
	Let $R \cong R_1 \times \ldots \times R_r$ with local rings $R_i$. Then
	the number of zero-divisors of $R$, i.e. the number of vertices of the zero-divisor graph $\Gamma(R)$ equals 
	\begin{equation*}
		\begin{split}
			\# V(\Gamma(R)) &= \prod_{i=1}^r \# R_i - \prod_{i=1}^r \#V(\ug(R_i)) -1 \\
			&=  \prod_{i=1}^r \# R_i - \prod_{i=1}^r \#U(R_i) -1.
		\end{split}
	\end{equation*}
\end{theorem}
\begin{proof}
	We have that $\#V(\eg(R_i)) = \#Z^*(R_i) + \#U(R_i) + 1 = \#R_i$ since $R_i = Z^*(R_i) \cup U(R_i) \cup \{0\}$. Taking into account that $\eg(R) \cong \eg(R_1) \times \ldots \times \eg(R_r)$, we therefore get that $\#V(\eg(R)) = \prod_{i=1}^r \#R_i$. Finally, since $\Gamma(R)$ arises from $\eg(R)$ by removing the vertex $0 \in R$ and all units of $R$ (where each unit of $R$ is a direct product of units of the $R_i$'s), the statement follows.
\end{proof}

Moreover, we get a similar result for the number of vertices of the compressed zero-divisor graph:

\begin{theorem} \label{thm:verticesCZDG}
	Let $R \cong R_1 \times \ldots \times R_r$ with local rings $R_i$. Then
	the number of vertices of the compressed zero-divisor graph $\Gamma_E(R)$ equals 
	\begin{equation*}
		\begin{split}
			\# V(\Gamma_E(R)) &= \prod_{i=1}^r \big(\# V(\Gamma_E({R_i})) + 2 \big)  -2 \\
			&= \prod_{i=1}^r \#{R_i}_E -2.
		\end{split}
	\end{equation*}
\end{theorem}
\begin{proof}
	Since the $R_i$'s are finite rings, each element of $R_i$ is either a zero-divisor or a unit. Thus, the elements of ${R_i}_E$ are exactly the vertices of $\Gamma_E(R_i)$ together with $[0]_{R_i}$ and $[1]_{R_i}$ (since the elements of $[1]_{R_i}$ are exactly the units of $R_i$). The statement follows from the construction of $\times_{\Gamma}$.
\end{proof}

From this theorem, we can immediately deduce the following nice (algebraic) corollary:
\begin{corollary} \label{cor:local}
	If $\#V(\Gamma_E(R)) + 2$ is a prime number, then $R$ is a local ring. Conversely, if $R$ is a local ring, then $\#V(\Gamma_E(R)) + 2$ is a prime power.
\end{corollary}

Note that the converse of the second statement is not true in general, i.e. $\#V(\Gamma_E(R)) + 2$ being a prime power does not imply that $R$ is a local ring, see for example $R \cong \mathbb Z_2 \times \mathbb Z_2$.

\begin{example}
		Let $R = \mathbb Z_3[[X,Y]]/(XY,X^3,Y^3,X^2-Y^2)$. The corresponding compressed zero-divisor graph $\Gamma_E(R)$ has 5 vertices, see Figure~\ref{fig:5}. Corollary~\ref{cor:local} shows that, therefore, $R$ has to be a local ring.
	
		\begin{figure}[H]
		\includegraphics[width=0.6\textwidth]{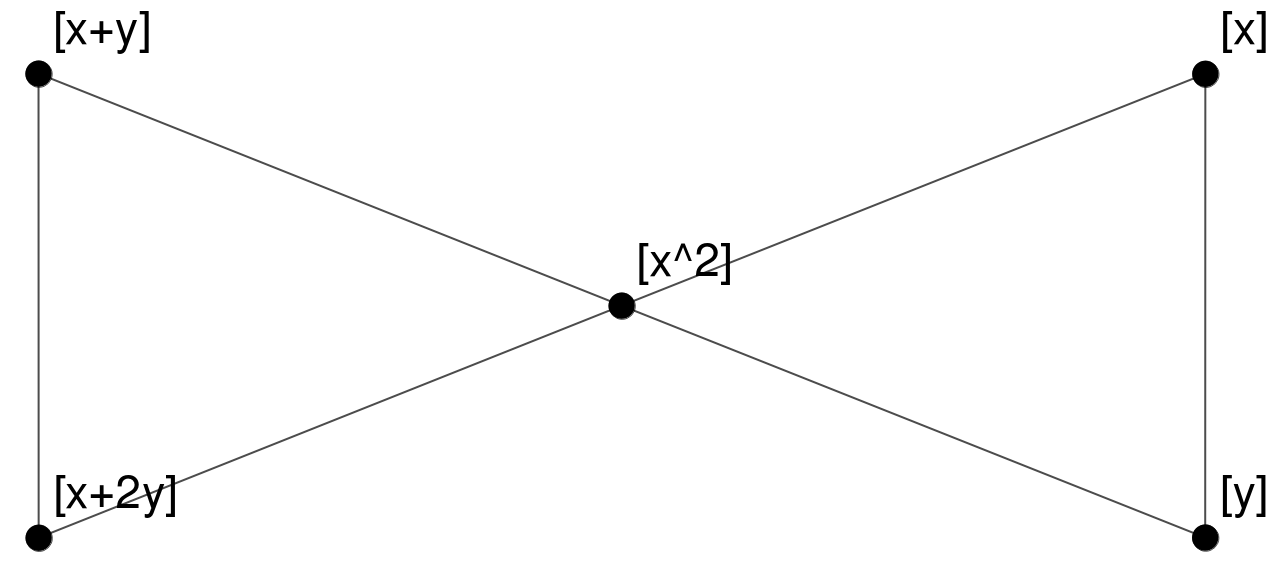}
		\caption{Compressed Zero-divisor graph $\Gamma_E(R)$ for $R=\mathbb Z_3[[X,Y]]/(XY,X^3,Y^3,X^2-Y^2)$.}
		\label{fig:5}
	\end{figure}
\end{example}

Finally, we end up with:

\begin{theorem} \label{thm:nullity1}
	Let $R \cong R_1 \times \ldots \times R_r$ with local rings $R_i$.  Then the nullity of the zero-divisor graph $\Gamma(R)$ is at least $$\eta(\Gamma(R)) \geq \prod_{i=1}^r \# R_i - \prod_{i=1}^r \#U(R_i)  -  \prod_{i=1}^r \big(\# V(\Gamma_E({R_i})) + 2 \big) + 1.$$
\end{theorem}
\begin{proof}
	Each element  of $[r]_R \in V(\Gamma_E(R))$ contributes exactly the same row to the adjacency matrix $A(\Gamma(R))$. Thus, $\rank A(\Gamma(R)) \leq \#Z^*(R_E) = \# V(\Gamma_E(R))$. Since $\eta(\Gamma(R)) = \dim A(\Gamma(R)) - \rank(A(\Gamma(R)))$ and $\dim A(\Gamma(R)) =  \#Z^*(R) = \#V(\Gamma(R))$, we have that $\eta(\Gamma(R)) \geq \#V(\Gamma(R)) - \#V(\Gamma_E(R))$. The statement follows with Theorems~\ref{thm:verticesZDG}~ and~\ref{thm:verticesCZDG}. 
\end{proof}

\section{Spectra of zero-divisor graphs of direct products of rings of integers modulo $n$} \label{sec:integers}
As already observed by Young~\cite{Young2015}, the adjacency matrix of the compressed zero-divisor graph is a so-called \textit{equitable partition} of the adjacency matrix of $\Gamma(R)$. A formal definition for this is given in \cite{Brouwer2012}. We define the \textit{weighted adjacency matrix} $\mathcal A(\Gamma_E(R))$ of the compressed zero-divisor graph as the matrix with $(i,j)$-th entry $$\mathcal A(\Gamma_E(R))_{i,j} = \begin{cases}
0, & \text{if}~ A(\Gamma_E(R))_{i,j} = 0, \\
\# [j]_R, & \text{else.}
\end{cases}$$
From \cite[Lemma 2.3.1]{Brouwer2012} it follows that every eigenvalue of $\mathcal A(\Gamma_E(R))$ is also an eigenvalue of $A(\Gamma(R))$. In general, it is not clear whether these eigenvalues are exactly the non-zero eigenvalues of $\Gamma(R)$, i.e. whether $\mathcal A(\Gamma_E(R))$ always has full rank. But assuming that $R$ is a product of rings of integers modulo $n$, we can prove the following:

\begin{theorem} \label{thm:rank}
	Let $R \cong \mathbb Z_{{p_1}^{t_1}} \times \ldots \times \mathbb Z_{{p_r}^{t_r}}$ for prime numbers $p_j$ and $r,t_j \in \mathbb N$. Then,
	$$\rank A(\Gamma(R)) = \rank \mathcal A (\Gamma_E(R)) = \# V(\Gamma_E(R)).$$
\end{theorem}
\begin{proof}
	We can easily see that $\rank A(\Gamma(R)) =\rank \mathcal A (\Gamma_E(R)) $ since for every $r \in R$, each element of $[r]_R$ contributes exactly the same row to the adjacency matrix $A(\Gamma(R))$. Thus, it suffices to show that $\rank \mathcal A (\Gamma_E(R)) = \# V(\Gamma_E(R))$. The matrix $\mathcal A(\Gamma_E(R_i))$ for $R_i = \mathbb Z_{{p_i}^{t_i}}$ is of the form $$\begin{pmatrix}
	0 & 0 & 0 & \cdots & 0 & 0 & p_i-1 \\
	0 & 0 & 0 &  \cdots  & 0 & p_i(p_i-1) & p_i-1 \\
	0 & 0 & 0 & \cdots  & p_i^2(p_i-1) & p_i(p_i-1) & p_i-1 \\
	\vdots & \vdots & \vdots &  \iddots  & \vdots &\vdots & \vdots \\
	0 & 0 & p_i^{t_i-4}(p_i-1) & \cdots & p_i^2(p_i-1) & p_i(p_i-1) & p_i-1 \\
	0 & p_i^{t_i-3}(p_i-1) & p_i^{t_i-4}(p_i-1) &\cdots & p_i^2(p_i-1) & p_i(p_i-1) & p_i-1 \\
	p_i^{t_i-2}(p_i-1) & p_i^{t_i-3}(p_i-1) & p_i^{t_i-4}(p_i-1) &\cdots & p_i^2(p_i-1) & p_i(p_i-1) & p_i-1 \\
	\end{pmatrix}$$ since $V(\Gamma_E(R_i)) = \{[p_i]_{R_i},[p_i^2]_{R_i},\ldots,[p_i^{t_i-1}]_{R_i}\}$ and $$\#[p_i^k]_{R_i} =  \#\{x \mid \gcd(x,p_i^{t_i}) = p_i^k \} =\varphi({p_i}^{t_i}/p_i^k) = {p_i}^{t_i-k-1}({p_i}-1).$$ Obviously, this matrix has full rank $\#V(\Gamma_E(R_i))$. Now, the graph $\eg_E(R_i)$ arises from $\Gamma_E(R_i)$ by adding the vertices $[1]_{R_i}$ and $[0]_{R_i}$ and the edges $\{[1]_{R_i},[0]_{R_i}\}$,$\{[0]_{R_i},[r]_{R_i}\}$ for all $[r]_{R_i} \in V(\Gamma_E(R_i))$. With an appropriate enumeration of the vertices of $\eg_E(R_i)$, it follows that the matrix $\mathcal A(\eg_E(R_i))$ equals $$\begin{pmatrix}
	0 & \begin{matrix}0 & \ldots & 0 \end{matrix} & 1 \\ \begin{matrix}
	0 \\ \vdots \\ 0
	\end{matrix} &
	\text{\Huge{ $\mathcal A(\Gamma_E(R_i))$}}
	& \begin{matrix}
	 1 \\ \vdots \\ 1
	 \end{matrix} \\
	 p_i^{t_i-1}(p-1) & \begin{matrix} p_i^{t_i-2}(p-1) & \ldots &p_i-1 \end{matrix} & 1
	\end{pmatrix}.$$ This matrix has full rank, too. Since $\eg_E(R) \cong \eg_E(R_1) \times \ldots \times \eg_E(R_r)$, the matrix $\mathcal A(\eg_E(R))$ equals the Kronecker product $\mathcal A(\eg_E(R_1)) \otimes \ldots \otimes \mathcal A(\eg_E(R_r))$ which has the form 
	$$\begin{pmatrix}
	0 & \begin{matrix}0 & \ldots & 0 \end{matrix} & 1 \\ \begin{matrix}
	0 \\ \vdots \\ 0
	\end{matrix} &
	\text{\Huge{ $\mathcal A(\Gamma_E(R))$}}
	& \begin{matrix}
	1 \\ \vdots \\ 1
	\end{matrix} \\
	x_1 & \begin{matrix} x_2 & \ldots &x_{\# V(\Gamma_E(R))+1}\end{matrix} & 1
	\end{pmatrix}$$
	for non-zero entries $x_j$. By the fact that the rank of the Kronecker product of two matrices equals the product of the ranks of these two matrices, we finally conclude that $\mathcal A(\eg_E(R))$, and therefore also $\mathcal A(\Gamma_E(R))$ has full rank, i.e. $\rank \mathcal A(\Gamma_E(R)) = \# V(\Gamma_E(R))$.
\end{proof}

With this result, we are able to improve Theorem~\ref{thm:nullity1}:

\begin{theorem} \label{thm:nullity2}
	Let $R \cong \mathbb \mathbb Z_{{p_1}^{t_1}} \times \ldots \times \mathbb Z_{{p_r}^{t_r}}$ for prime numbers $p_j$ and $r,t_j \in \mathbb N$. Then the zero-divisor graph $\Gamma(R)$ has $$\prod_{i=1}^r (t_i+1)-2$$ non-zero eigenvalues, and the nullity of $\Gamma(R)$ equals $$\eta(\Gamma(R)) = \prod_{i=1}^r p_i^{t_i-1} \Big( \prod_{i=1}^r p_i - \prod_{i=1}^r (p_i-1) \Big) - \prod_{i=1}^r (t_i+1) +1.$$
\end{theorem}
\begin{proof}
	By Theorem~\ref{thm:rank}, the number of non-zero eigenvalues of $\Gamma(R)$ equals the number of vertices of the compressed zero-divisor graph $\Gamma_E(R)$. Since $V(\Gamma_E(R_i)) = \{[p_i]_{R_i},[p_i^2]_{R_i},\ldots,[p_i^{t_i-1}]_{R_i}\}$, we deduce form Theorem~\ref{thm:verticesCZDG} that $$\#V(\Gamma_E(R)) = \prod_{i=1}^r (t_i+1)-2.$$ Similar as in the proof of Theorem~\ref{thm:nullity1}, we see that $\eta(\Gamma(R)) = \#V(\Gamma(R)) - \#V(\Gamma_E(R))$. The number of units in $R_i$ ($ = \#U(R_i) = \#V(\ug(R_i))$) is given by $\varphi(p_i^{t_i}) = p_i^{t_i-1}(p_i-1)$. Thus, by Theorem~\ref{thm:verticesZDG}, we get that 
	\begin{equation*}
		\begin{split}
			\#V(\Gamma(R)) &= \prod_{i=1}^r p_i^{t_i} - \prod_{i=1}^r p_i^{t_i-1}(p_i-1) -1\\
			&= \prod_{i=1}^r p_i^{t_i-1} \Big( \prod_{i=1}^r p_i - \prod_{i=1}^r (p_i-1) \Big) - 1,
		\end{split}
	\end{equation*} and, therefore, the statement follows.
\end{proof}
Note that the number of non-zero eigenvalues of $\Gamma(\mathbb Z_{{p_1}^{t_1}} \times \ldots \times \mathbb Z_{{p_r}^{t_r}})$ does not depend on the prime numbers $p_i$ but on the powers $t_i$ only.

Now, we can easily determine the eigenvalues of $\Gamma(R)$ for $R \cong \mathbb Z_{{p_1}^{t_1}} \times \ldots \times \mathbb Z_{{p_r}^{t_r}}$. Theorem~\ref{thm:nullity2} gives us the number of eigenvalues equal to zero. The non-zero eigenvalues can be computed as in the proof of Theorem~\ref{thm:rank}. We illustrate this in the following examples. Note that the eigenvalues of the graphs $\Gamma(\mathbb Z_{p^2}), \Gamma(\mathbb Z_{p^3}), \Gamma(\mathbb Z_{p} \times \mathbb Z_{q})$ and $\Gamma(\mathbb Z_{p^2} \times \mathbb Z_{q})$ for prime numbers $p \neq q$ were already determined by Young~\cite{Young2015} resp. Surendranath Reddy et. al.~\cite{Reddy2017}, and the ones of $\Gamma(\mathbb Z_p \times \mathbb Z_p)$ by Sharma et. al.~\cite{Sharma2011}.

\begin{example} \label{ex:1}
		Let $p$ be a prime number and $R \cong \mathbb Z_p \times \mathbb Z_p \times \mathbb Z_p$. By Theorem~\ref{thm:nullity2} the multiplicity of the eigenvalue 0 of $\Gamma(R)$ equals $$\eta(\Gamma(R)) = (p^3-(p-1)^3)-2^3+1 = 3(p+1)(p-2).$$ The ring $\mathbb Z_p$ has no zero-divisors and, therefore, $\Gamma(\mathbb Z_p)$ is the empty graph. Thus, the matrix $\mathcal A(\eg_E(\mathbb Z_p))$ is given by $$\mathcal A(\eg_E(\mathbb Z_p)) = \begin{pmatrix}
		0 & 1 \\ p-1 & 1
		\end{pmatrix}.$$ Now, we compute the Kronecker product $$\mathcal A(\eg_E(\mathbb Z_p)) \otimes \mathcal A(\eg_E(\mathbb Z_p)) \otimes \mathcal A(\eg_E(\mathbb Z_p))$$ which yields the matrix $$\begin{pmatrix}
			0 & 0 & 0 & 0 & 0 & 0 & 0 & 1 \\
			0 & 0 & 0 & 0 & 0 &0 & p-1 & 1 \\
			0 & 0 & 0 & 0 & 0 & p-1 & 0 & 1 \\
			0 & 0 & 0 & 0 & (p-1)^2 & p-1 & p-1 & 1 \\
			0 & 0 & 0 & p-1 & 0 & 0 & 0 & 1 \\
			0 & 0 & (p-1)^2 & p-1 & 0 & 0 & p-1 & 1 \\
			0 & (p-1)^2 & 0 & p-1 & 0 & p-1 & 0 & 1 \\
			(p-1)^3 & (p-1)^2 & (p-1)^2 & (p-1) & (p-1)^2 & p-1 & p-1 & 1
		\end{pmatrix}.$$ Hence, $\mathcal A (\Gamma_E(R))$ equals the submatrix  $$\begin{pmatrix}
		0 & 0 & 0 & 0 &0 & p-1 \\
		0 & 0 & 0 & 0 & p-1 & 0 \\
		0 & 0 & 0 & (p-1)^2 & p-1 & p-1 \\
		0 & 0 & p-1 & 0 & 0 & 0 \\
		0 & (p-1)^2 & p-1 & 0 & 0 & p-1  \\
		(p-1)^2 & 0 & p-1 & 0 & p-1 & 0  
		\end{pmatrix},$$ which has characteristic polynomial $$\chi_{\Gamma(R)}(x)  = -(-1 + 3 p - 3 p^2 + p^3 + (1-p)x - x^2)^2 (-1 + 3 p - 3 p^2 + p^3 + 2(p-1)x - x^2).$$ The roots of this polynomial, i.e. the non-zero eigenvalues of $\Gamma(R)$, are $$\lambda_{1,2} = \frac{1}{2} \big(1-p \pm (p-1) \sqrt{4p-3} \big), \quad \lambda_{3,4} = p-1 \pm \sqrt{p-2p^2+p^3},$$ and, therefore, the spectrum of $\Gamma(R)$ equals $$\spec(\Gamma(R)) = \big\{ \lambda_1^{[2]}, \lambda_2^{[2]}, \lambda_3^{[1]}, \lambda_4^{[1]}, 0^{[3(p+1)(p-2)]} \big\}.$$
\end{example}

\begin{example}
	Let $p$ be a prime number and $R \cong \mathbb Z_p \times \mathbb Z_p \times \mathbb Z_p \times \mathbb Z_p$. By Theorem~\ref{thm:nullity2} the multiplicity of the eigenvalue 0 of $\Gamma(R)$ equals $$\eta(\Gamma(R)) = p^4-(p-1)^4-2^4-1.$$ Analogously as in Example~\ref{ex:1}, we find the matrix $\mathcal A (\Gamma_E(R))$ as a submatrix of the Kronecker product $$\mathcal A(\eg_E(\mathbb Z_p)) \otimes \mathcal A(\eg_E(\mathbb Z_p)) \otimes \mathcal A(\eg_E(\mathbb Z_p)) \otimes \mathcal A(\eg_E(\mathbb Z_p)).$$ The characteristic polynomial of this matrix is 
	\begin{equation*}
		\begin{split}
			\chi_{\Gamma(R)}(x) = &-(1 - 2 p + p^2 - x)^5 (1 - 2 p + p^2 + x) \times \\ &\times \big(1 - 4 p + 6 p^2 - 4 p^3 + 
			p^4 + (1+p-2p^2)x + x^2 \big) \times\\ & \times \big(1 - 4 p + 6 p^2 - 4 p^3 + p^4 + 
			(1-3p+2p^2) x + x^2 \big)^3
		\end{split}
	\end{equation*} and has roots $$\lambda_1 = (p-1)^2,\quad \lambda_2 = -p^2+p-1,$$ $$\lambda_{3,4} = \frac{1}{2} \big(- 2 p^2 +3p -1 \pm (p-1) \sqrt{4 p-3} \big),$$ $$\lambda_{5,6} = \frac{1}{2} \big(2 p^2 -p -1 \pm \sqrt{3} \sqrt{4 p^3-9p^2+6p-1} \big).$$
	Hence, the spectrum of $\Gamma(R)$ is given by $$\spec(\Gamma(R)) = \big\{\lambda_1^{[5]},\lambda_2^{[1]},\lambda_3^{[3]},\lambda_4^{[3]},\lambda_5^{[1]},\lambda_6^{[1]}, 0^{[p^4-(p-1)^4-2^4-1]} \big\}.$$
\end{example}

\begin{example}
	Unfortunately, if we consider not only products of the ring $\mathbb Z_p$ but also of rings of the form $\mathbb Z_{p^t}$ for $t>1$ or of the form $\mathbb Z_q$ for a prime $q \neq p$, the eigenvalues of $\Gamma(R)$ get very cumbersome. However, at least we want to include the characteristic polynomials of the graphs $\Gamma(\mathbb Z_p \times \mathbb Z_p \times \mathbb Z_q)$ and $\Gamma(\mathbb Z_{p^2} \times \mathbb Z_p)$. Note that $$\mathcal A(\eg_E(\mathbb Z_p)) = \begin{pmatrix}
		0 & 0 & 1 \\ 0 & p-1 & 1 \\ p(p-1) & p-1 & 1
	\end{pmatrix}.$$ With the same method as in the examples before, we find the polynomials 
	\begin{equation*}
		\begin{split}
			\chi_{\mathbb Z_p \times \mathbb Z_p \times \mathbb Z_q}(x) = &-(p-1)^6 (q-1)^3 + (p-1)^3 (q-1) \big(p (3 q-2)-q \big) x^2 -\\ & - 
			2 (p-1)^2 (q-1) x^3 - (p-1) \big(p (3 q-2)-q \big) x^4 + x^6,
		\end{split}
	\end{equation*}
	\begin{equation*}
		\begin{split}
			\chi_{\mathbb Z_{p^2} \times \mathbb Z_p}(x) = (p-1)^5 p + (p-1)^3 p x - 
			2 (p-1)^2 p x^2 - (p-1) x^3 + x^4.
		\end{split}
	\end{equation*}
\end{example}

\begin{remark}
	It is clear that two rings are isomorphic only if their respective zero-divisor graphs are isomorphic. Moreover, two graphs are isomorphic only if they have the same characteristic polynomial. Thus, in order to see that two rings are non-isomorphic, it might help to compare the characteristic polynomials of their corresponding zero-divisor graphs. 
\end{remark}

\section{A relation between $\chi_{\Gamma(R)}$ and $\chi_{\eg(R)}$} \label{sec:relation}
The main interest in spectral graph theory of building graph products is that there are relations between the eigenvalues of graphs and the eigenvalues of their product. For example, we already mentioned that if $\lambda_i, \mu_i$ denote the eigenvalues of graphs $G_1$ resp. $G_2$, then the eigenvalues of the direct product $G_1 \times G_2$ are exactly the values $\lambda_i \mu_j$. Similar results are also known for the point identification and the complete product of simple graphs: let $G-v$ denote the graph arising from removing the vertex $v$ of $G$ together with all its edges, and let $\overline{G}$ be the complement of $G$ (that is, the graph with same vertex set as $G$, where two distinct vertices are adjacent whenever they are non-adjacent in $G$), then the following holds:

\begin{lemma} \label{lemma:1}
	Let $G_1$ and $G_2$ be simple graphs with $v \in V(G_1)$ and $w \in V(G_2)$. The point-identification $v=w$ yields 
	\begin{equation*}
		\begin{split}
			\chi_{G_1 \bullet G_2}(x) = 
			\chi_{G_1}(x) \chi_{G_2-w}(x) &+ \chi_{G_1-v}(x) \chi_{G_2}(x)\\ &- x \chi_{G_1-v}(x)\chi_{G_2-w}(x).
		\end{split}
	\end{equation*}
\end{lemma}

\begin{lemma} \label{lemma:2}
	Let $G_1$ and $G_2$ be simple graphs with $\#V(G_1) = n_1$ and $\#V(G_2) = n_2$. Then the characteristic polynomial of the complete product of $G_1$ and $G_2$ equals 
	\begin{equation*}
		\begin{split}
			\chi_{G_1 \nabla G_2}(x) = 
			(-1)^{n_2}\chi_{G_1}(x) \chi_{\overline{G_2}}(-x-1) &+ (-1)^{n_1}\chi_{G_2}(x) \chi_{\overline{G_1}}(-x-1)\\ &- (-1)^{n_1+n_2} \chi_{\overline{G_1}}(-x-1)\chi_{\overline{G_2}}(-x-1).
		\end{split}
	\end{equation*}
\end{lemma}

Proofs for these lemmas are given in \cite{Cvetkovic1980}. We can easily see that the formula in Lemma~\ref{lemma:1} still holds for graphs with loops if the graphs do not have loops on both vertices, $v$ and $w$. Moreover, Hwang~and~Park~\cite{Hwang2011} generalized the result of Lemma~\ref{lemma:2}:

\begin{lemma} \label{lemma:2.2}
	Let $A \in \mathbb R^{m \times m}, B \in \mathbb R^{n \times n}, a,c \in \mathbb R^m, b,d \in \mathbb R^n,$ $$M = \begin{pmatrix}
	A & ad^{\sf t} \\ bc^{\sf t} & B
	\end{pmatrix}$$ and $\tilde{A} = ac^{\sf t} -A, \tilde{B}=bd^{\sf t}-B$. Then \begin{equation*}
		\begin{split}
			\chi_{M}(x) = 
			(-1)^{m}\chi_{\tilde{A}}(-x) \chi_{B}(x) &+ (-1)^{n}\chi_{A}(x) \chi_{\tilde{B}}(-x)\\ &- (-1)^{m+n} \chi_{\tilde{A}}(-x)\chi_{\tilde{B}}(-x).
		\end{split}
	\end{equation*}
\end{lemma}

Therefore, in the following let $\overline{G}$ denote the \textit{generalized complement} of $G$, i.e. the graph with same vertex set as $G$, where two not necessarily distinct vertices are adjacent whenever they are non-adjacent in $G$. That is, the graph with adjacency matrix $A(\overline{G}) = J - A(G)$, where $J$ denotes the all-1 matrix. Now, we are able to prove the following:
\begin{theorem} \label{thm:charpoly}
	Let $R$ be a finite commutative ring and $n = \#U(R)$, i.e. the number of units in $R$. Then we have that $$\chi_{\eg(R)}(x) = x^{n-1} \big( (-1)^{n+1} \chi_{\overline{\Gamma(R)}}(-x) x + \chi_{\Gamma(R)}(x)(x^2-n) \big).$$
\end{theorem}
\begin{proof}
	We recall that $$\eg(R) = \big (\Gamma(R) \nabla \zg(R) \big) \overset{\{0\}}{\bullet} \big(\ug(R) \nabla \zgl(R) \big).$$ We first determine the characteristic polynomial of $\ug(R) \nabla \zgl(R)$ by applying Lemma~\ref{lemma:2.2} for $A= \begin{pmatrix}
	1
	\end{pmatrix}$ and $B$ being the zero-matrix of dimension $n \times n$. We can easily see that $\chi_A(x) = x-1$, $\chi_{\tilde A}(x) = x$, $\chi_B(x) = x^n$ and $\chi_{\tilde B}(x) = x^{n-1}(x-n)$. Thus, we get \begin{equation*}
		\begin{split}
			\chi_{\ug(R) \nabla \zgl(R)}(x) &= \chi_M(x) \\ &= (-1)(-x)x^n+(-1)^n(x-1)(-x)^{n-1}(-x-n) -\\ & \quad \quad - (-1)^{n+1}(-x)(-x)^{n-1}(-x-n) \\ &= x^{n-1}(x^2-x-n).
		\end{split}
	\end{equation*}
Analogously, we find the characteristic polynomial of $\Gamma(R) \nabla \zg(R)$ for $A = \begin{pmatrix}
0
\end{pmatrix}$ and $B = A(\Gamma(R))$ to be \begin{equation*}
	\begin{split}
		\chi_{\Gamma(R) \nabla \zg(R)}(x) = (-1)^{n+1} \chi_{\overline{\Gamma(R)}}(-x) + \chi_{\Gamma(R)}(x)(x+1).
	\end{split}
\end{equation*}

Finally, with Lemma~\ref{lemma:1} we get \begin{equation*}
	\begin{split}
		\chi_{\eg(R)}(x) &= \chi_{\ug(R) \nabla \zgl(R)}(x) \chi_{\Gamma(R)} + x^n \chi_{\Gamma(R) \nabla \zg(R)}(x) - x \cdot x^n\chi_{\Gamma(R)} \\ &=x^{n-1}(x^2-x-n)\chi_{\Gamma(R)}(x) + \\ & \quad \quad + x^n \big((-1)^{n+1} \chi_{\overline{\Gamma(R)}}(-x) + \chi_{\Gamma(R)}(x)(x+1) \big) - \\ & \quad \quad \quad \quad - x^{n+1}\chi_{\Gamma(R)}(x) \\ &= x^{n-1} \big( (-1)^{n+1} \chi_{\overline{\Gamma(R)}}(-x) x + \chi_{\Gamma(R)}(x)(x^2-n) \big).
	\end{split}
\end{equation*}
\end{proof}

\begin{remark}
	If $R \cong R_1 \times \ldots \times R_r$, we can apply Theorem~\ref{thm:charpoly} to each of the rings $R_i$, which gives us the characteristic polynomials $\chi_{\eg(R_i)}$. By computing the roots of $\chi_{\eg(R_i)}$, we find the eigenvalues of $\eg(R)$ to be all possible products of these roots, since $\eg(R) \cong \eg(R_1) \times \ldots \times \eg(R_r)$. Unfortunately, it is difficult to extrapolate the eigenvalues of $\Gamma(R)$ from the ones of $\eg(R)$, since the characteristic polynomial $\chi_{\Gamma(R)}$ not only depends on $\chi_{\eg(R)}$, but also on the characteristic polynomial of the generalized complement of $\Gamma(R)$.
\end{remark}

 \newcommand{\noop}[1]{}

\end{document}